\documentclass[11pt,a4paper]{amsart}

\usepackage{amssymb}
\usepackage{amsthm}
\usepackage{amsmath}
\usepackage{amstext}

\numberwithin{equation}{section}

\theoremstyle{plain}
\newtheorem{thm}{Theorem}
\newtheorem{prop}[thm]{Proposition}

\newtheorem*{notation}{Notation}

\newcommand{\C}{\mathbb{C}}
\newcommand{\R}{\mathbb{R}}

\newcommand{\p}{\partial}
\DeclareMathOperator{\grad}{\p}

\def\({(\!(}
\def\){)\!)}

\title[CHARACTERIZATION OF ISOLATED HOMOGENEOUS HYPERSURFACE SINGULARITIES]{THE YAU'S CHARACTERIZATION OF
ISOLATED HOMOGENEOUS HYPERSURFACE SINGULARITIES}

\author{Ould M Abderrahmane}

\address{D\'epartement de Math\'ematiques, Universit\'e des Sciences,
de Technologie et de Médecine BP. 880, Nouakchott, Mauritanie}

\email{ymoine@ustm.mr} 

\newcommand{\abstracttext}{Using the semicontinuity of the Milnor
number, the Kouchnirenko's theorem \cite{AGK} on the Newton number
and the geometric characterization of $\mu$-constancy in \cite
{{greuel},{LS}, {BT1}}, we give a new proof of Yau's
characterization of isolated homogeneous hypersurface
singularities.}

\begin{document}
\begin{abstract} \abstracttext \end{abstract} \maketitle

\bigskip
\section{Introduction}
\bigskip

 Let $f \colon (\C^n, 0) \rightarrow (\C, 0)$ be the germ of a
holomorphic function with an isolated singularity. The Milnor number
of a germ $f$, denoted by $\mu(f)$, is algebraically defined as the
$\text{dim}\,\mathcal{O}_n/{J(f)}$, where $\mathcal{O}_n$ is the
ring of complex analytic function germs $\colon (\C^n, 0) \to (\C,
0)$ and  $J(f) $ is the Jacobian ideal in $\mathcal{O}_n$ generated
by the partial derivatives $\{\frac{\grad f}{\grad
z_1},\cdots,\frac{\grad f}{\grad z_n}\}$, and the Tjurina number
$\tau$ of the singularity $(V, 0)$ are defined by
$\text{dim}\,\mathcal{O}_n/{Tr(f)}$, where $Tr(f)=\langle
f\rangle+J(f) $ is the   Tjurina ideal of $f$ in $\mathcal{O}_n$.
 We recall that
 the multiplicity $m(f)$
is defined  as the lowest degree in the power series expansion of
$f$ at $0\in \C^n$.

Moreover, $f$ is called weighted homogeneous of degree $d$ with
respect the weight $w=(w_1,\dots,w_n)$ (or type $(d; w)$) if $f$ may
be written as a sum of monomials $z_1^{\alpha_1}\cdots
z_n^{\alpha_n}$ with
\begin{equation}\label{weight}
\alpha_1 w_1 +\dots +\alpha_n w_n =d.
\end{equation}

It is a natural question to ask when $V$ can be defined by a
weighted homogeneous polynomial or a homogeneous polynomial up to a
biholomorphic change of coordinates. The former question was
answered in a celebrated  paper by Saito in $1971$ \cite{KS}.
However, the latter question had remained open for $40$ years until
Xu and Yau solved it for $f$ with three variables \cite{XY2}. Then
Yau and Zuo \cite{YZ1} solved it for $f$ with up to six variables.
Recently, Yau and  Zuo \cite{YZ2} solve the latter question
completely; i.e., they show that $f$ is a homogeneous polynomial
after a biholomorphic change of coordinates if and only if $\mu=
\tau = (m - 1)^n$, where $\mu$, $\tau$ and $m$ are, the Milnor
number, the Tjurina number and the multiplicity of the singularity
respectively.

The aim of this is paper is to give an alternative proof of the
result of Yau and Zuo \cite{YZ2}. Our proof is simple and
elementary. The main tool is the semicontinuity of the Milnor
number,  the Kouchnirenko's theorem \cite{AGK} on the Newton number
and the geometric characterization of $\mu$-constancy in \cite
{{greuel},{LS}, {BT1}}. First, we give an alternative proof of the
first part of Yau conjecture $1.7$ in \cite{LWY}(Proposition
\ref{prop} below). By using it and the splitting lemma, we obtain
the Yau's characterization of isolated homogeneous hypersurface
singularities.

\bigskip


The purpose of this paper is to prove the following results:

\begin{thm}[Yau conjecture $1.7$ in \cite{LWY}]\label{main}

$(1)$ Let $f :(\C^n,0) \to (\C,0)$ be a holomorphic germ defining an
isolated hypersurface singularity $V = \{z \in\C^n\; :\; f(z)=0\}$
at the origin. Let $\mu$ and $m$ be the Milnor number and
multiplicity of $(V,0)$, respectively. Then
\begin{equation}\label{1.1}
\mu \geq (m - 1)^n,
\end{equation}
 and the equality in $(\ref{1.1})$ holds if and only if
$f$ is a semi-homogeneous function (i.e., $f = f_m + g$, where $f_m$
is a nondegenerate homogeneous polynomial of degree $m$ and $g$
consists of terms of degree at least $m +1$).

$(2)$ Suppose that $f$ is a weighted homogenous function. Then the
equality in $(\ref{1.1})$ holds if and only if $f$ is a homogeneous
polynomial (after a biholomorphic change of coordinates).
\end{thm}

\begin{notation}
To simplify the notation, we will adopt the following conventions\,:
for a function $F(z, t)$(for $(z, t)\in\C^n\times \C^m$) we denote
by $\grad F$ the gradient of $F$ and by $\grad_z F$ the gradient of
$F$ with respect to variables $z$. Also, we denote by $ \| \grad
F\|^2=\| \grad_z F\|^2+\| \grad_t F\|^2=\sum_{ i=1}^n |\frac{\grad
F}{\grad z_i} |^2+\sum_{ j=1}^m |\frac{\grad F}{\grad t_j} |^2. $

Let $\varphi,\,\, \psi \colon(\C^n, 0) \to \R$ be two function
germs. We say that $ \varphi(x)\lesssim \psi(x)$ if there exists a
positive constant $C> 0$ and an open neighborhood $U$ of the  origin
in $\C^n$ such that $\varphi(x)\leq C \; \psi(x)$, for all $x \in
U$. We write $ \varphi(x) \sim \psi(x)$  if $\varphi(x) \lesssim
\psi(x)$ and $\psi(x)\lesssim \varphi(x)$. Finally,
$|\varphi(x)|\ll|\psi(x)|$ (when $x$ tends to $x_0=0$) means
$\lim_{x\to x_0}\frac{\varphi(x)}{\psi(x)}=0$.

\end{notation}

\bigskip
\section{Proof of the  theorem \ref{main}}
 Before starting the proofs, we will recall some
important results on the Newton number and the geometric
characterization of $\mu$-constancy.

\begin{thm}[Greuel \cite{greuel}, L\^e-Saito \cite{LS}, Teissier \cite{BT1}]\label{le-saito-tessier}
Let $F \colon (\C^n\times \C^m, 0) \to (\C, 0)$ be the deformation
of a holomorphic $f \colon (\C^n, 0) \to (\C, 0)$ with isolated
singularity. The following statements are equivalent.
\begin{enumerate}
\item $F$ is a $\mu$-constant deformation of $f$.

\item $\frac{\grad F}{\grad t_j}\in \overline{J(F_t)}$, where $\overline{J(F_t)}$
denotes the integral closure of the Jacobian ideal of $F_t$
generated by the partial derivatives of $F$ with respect to the
variables $z_1,\dots, z_n$.

\item  The deformation $ F(z, t) = F_t(z)$ is a Thom map, that is,
$$
\sum_{ j=1}^m |\frac{\grad F}{\grad t_j} | \ll \| \grad F\| \text{
as } (z, t) \to (0, 0).
$$

\item  The polar curve of $F$ with respect to $\{t = 0\}$ does not split,
that is,
$$
\{(z,t) \in \C^n \times \C^m \,\,| \,\,\grad_zF (z,t) = 0\} = \{0\}
\times \C^m \text{ near } (0,0).
$$

\end{enumerate}
\end{thm}

Firstly, we  want to show the first part of the main theorem.

\begin{prop}[Teissier \cite{BT1}, Furuya-Tomari\cite{Furuya-Tomai}]\label{prop}
 Let $f :(\C^n,0) \to (\C,0)$ be a holomorphic
germ defining an isolated hypersurface singularity $V = \{z
\in\C^n\; :\; f(z)=0\}$ at the origin. Let $\mu$ and $m$ be the
Milnor number and multiplicity of $(V,0)$, respectively. Then
\begin{equation}\label{p1}
\mu \geq (m - 1)^n,
\end{equation}
 and the equality in $(\ref{p1})$ holds if and only if
$f$ is a semi-homogeneous function.

\end{prop}

\begin{proof}
Let
$$
f=f_{m}+f_{m+1}+\cdots; f_{m}\neq 0,
$$
be the Taylor expansion of $f$ at the origin. Take the following
family of singularities
\begin{equation}\label{2.1}
F(z,t)=f_t(z)=f_{m}(z)+\sum_{i=1}^nt_i\,z_i^{m} +f_{m+1}+\cdots;
\quad t=(t_1,\cdots,t_n)  \text{ near } 0\in\C^n.
\end{equation}

It follows from the upper semicontinuity of Milnor number \cite{JM}
and the Kouchnirenko's result \cite{AGK} that
\begin{equation}\label{2.2}
\mu(f)\geq\mu(f_t)\geq\nu(\Gamma_-(f_t))=(m -1)^n,
\end{equation}

where $\nu(\Gamma_-(f_t))$ is the Newton number (see \cite{AGK} for
details).

Now we have to prove that equality in $(\ref{p1})$ holds if and only
if $f$ is a semi-homogeneous function.

$\Leftarrow$  Suppose $f$ is a semi-homogeneous function, so we get
by the theorem in (\cite{AGV} section $12.2$) that $\mu(f)=(m
-1)^n$.



$\Rightarrow$ Let $f_t$ the deformation of $f$ defined in
$(\ref{2.1})$, it follows from the upper semicontinuity of Milnor
number \cite{JM} and the $(\ref{2.2})$ that
$$
(m -1)^n=\mu(f)\geq\mu(f_t)\geq  (m -1)^n,\; \text{ for each }
t_i\neq 0, \;i=1,\cdots,n.
$$
Hence the deformation $f_t$ is $\mu$-constant, so by the theorem
\ref{le-saito-tessier} we get\, :
$$
\parallel\p_t F\parallel \sim\parallel
z\parallel^m\ll\parallel\p_z F\parallel\quad \text{ for }(z,t)\to
(0,0).
$$
Therefore for $t=0$, we obtain

\begin{equation}\label{2.3}
\parallel
z\parallel^m\ll\parallel\p_z f\parallel\quad \text{ for }z\to 0.
\end{equation}

But $\parallel\p_z f\parallel\leq\parallel\p_z f_m\parallel +
\parallel\p_z f_{m+1}+\p_z f_{m+2}+\dots\parallel$, it is easy to find that
$$
\parallel z\parallel^m\gtrsim \parallel\p_z f_{m+1}+\p_z
f_{m+2}+\dots\parallel \text{ for }z\to 0.
$$
It follows from (\ref{2.3}) that
$$
\parallel
z\parallel^m\ll\parallel\p_z f_m\parallel\quad \text{ for }z\to 0.
$$
Which implies that $\p_zf_m=0$ if and only if $ z=0$. So, we obtain
that $f_m$ has an isolated singularity at the origin.
  This end the proof of the first
part of the  main theorem.
\end{proof}

Finally, suppose that $f$ is a weighted homogeneous function, and we
intend to show that the equality in $(\ref{1.1})$ holds if and only
if $f$ is a homogeneous polynomial (after a biholomorphic change of
coordinates).

Since $(\Leftarrow)$ is trivial it is enough to see $(\Rightarrow)$.

Modulo a permutation coordinate of $\C^n$, we may assume that
\begin{equation}\label{;;;}
w_1\leq w_2\leq\cdots\leq w_n.
\end{equation}

 Let $f=f_{m}+f_{m+1}+\cdots;\text{ with } f_{m}\neq 0,$ be the
Taylor expansion of $f$ at the origin. Since $f$ is weighted
homogenous of degree type $d$.

If, we suppose that $m=2$, then we have $\mu(f)=(2-1)^n=1$, it
follows from the above proposition \ref{prop} that $f_2$ defining an
isolated singularity at the origin $0\in \C^n$, then there exists
the term $z_n^2$ or $z_nz_j$ for some integers $j\geq1$ with
non-zero coefficients in $f$, also $z_1^2$ or $z_1z_i$ for some
integers $i\geq1$ appears in expansion of $f$. There are two cases
to be considered.

{\bf Case 1.} In this case, we suppose that $z_n^2$ or $z_1^2$
appears in expansion of $f$, then we get
$$
d=2w_n \geq 2w_{n-1}\geq\cdots\geq 2w_1=d,
$$
or
$$
d=2w_n\geq w_n+w_i\geq w_1+w_i=d,
$$
or
$$
d=w_n+w_j\geq w_n+w_1\geq 2w_1=d.
$$
But this clearly implies
$$
w_1=w_2=\cdots=w_n.
$$

{\bf Case 2.} In this case, we suppose $z_nz_j$ and $z_1z_i$  for
some integers $i\geq 1$ and $j\geq 1$ with non-zero coefficients in
$f$, then
$$
d=w_n+w_j=w_1+w_i.
$$
 Moreover, for any
monomial $z_1^{\alpha_1}\cdots z_n^{\alpha_n}$ of $f$ with
${\alpha_n}\neq 0$, we have
$$
d=\sum_{j<n}\alpha_jw_j +\alpha_nw_n\geq
\left(\sum_{j<n}\alpha_j\right)w_1+ w_n\geq w_1 +w_i=d.
$$
Then,
\begin{equation}\label{2.7}
\sum_{j<n}\alpha_j=1,\quad {\alpha_n}=1 \quad \text{ and } w_i=w_n.
\end{equation}
 Therefore we
may write
$$
f(z)=a_jz_jz_n + \sum_{k\neq j} a_k z_kz_n+ f(z_1,\dots, z_{n-1},
0),\quad a_{j}\neq0.
$$
If $a_k\neq0$, then we have $d=w_n+w_k =w_n+w_j=w_1+w_i$, it follows
from $(\ref{2.7})$ that  $w_j=w_k=w_1$. After a permutation of
coordinates with same weights it can be written as
$$
f(z)=z_1z_n + \sum_{j>1} a_j z_jz_n+ f(z_1,\dots, z_{n-1}, 0).
$$
 Then we may assume, by a change of coordinates of the form
$\xi_{1}=z_{1} + \sum_{j>1} a_jz_j$, that $ f(z)=z_1z_n+ f(z_1,\dots
z_{n-1}, 0)=z_1(z_n + g(z))+f(0,z_2\dots,z_{n-1}, 0)$.
 Also by a change of
coordinates of the form $\xi_n=z_n +g(z)$, we can assume that
$$
f(z)=z_1z_n +f(0,z_2,\dots, z_{n-1},0).
$$
Continuing the same process on $f(0,z_2,\dots, z_{n-1},0)$, we can
deduce by   the splitting lemma,
 that  $f(z)=z_1^2+\cdots +z_n^2$ after a biholomorphic
change of coordinates.\\


From now, we  suppose that $m\geq 3$, it follows from the above
proposition \ref{prop}  that $f_{m}$ is an isolated singularity at
the origin.  Since $f_{m}$ is a homogeneous of degree $m$ with
isolated singularity, it is easy to check that the monomial
$z_1^{m}$ or $z_1^{m-1}z_i$  for some integers $i>1$ appears in the
expansion of $f$. There are two cases to be considered.

{\bf Case 1.} In this case, we suppose $z_1^{m}$ appears in the
expansion of $f_{m}$. Since $f_{m}$ is an isolated singularity at
the origin $0\in \C^n$, there exist the terms $z_n^{m}$ or
$z_n^{m-1}z_j$ with non-zero coefficients in $f$.

 From the hypotheses, we have
$$
d=m w_n\geq m w_{n-1}\geq\cdots \geq m w_1=d,
$$
or
$$
d=(m-1)w_n + w_j\geq (m-1)w_1 + w_j \geq (m-1) w_1 +w_1=d.
$$
Hence
$$
w_1=w_2=\dots=w_n.
$$
This end the proof of second part of the main theorem in the first
case.

 {\bf Case 2.} In this case, we suppose
$z_1^{m-1}z_i$ appears in the expansion of $f_{m}$, since $f_{m}$
defining an isolated singularity at the origin $0\in \C^n$, there
exist the terms $z_n^{m}$ or $z_n^{m-1}z_j$ with non-zero
coefficients in $f$.

From the hypotheses, we have
$$
d=m w_n\geq (m -1)w_n+ w_i\geq (m-1) w_1 +w_i=d,
$$
or
$$
d=(m-1)w_n + w_j\geq (m-2)w_n +w_i+ w_j\geq (m-2)w_1 +w_i+ w_j \geq
(m-2) w_1 +w_i+w_1=d.
$$
Hence
$$
w_1=w_2=\dots=w_n.
$$
This completes the proof of the  theorem.

\bigskip


\end{document}